\documentclass[10pt,a4paper]{amsart}

\usepackage{amssymb}
\usepackage{eucal}
\usepackage{hyperref}
\usepackage{microtype}
\usepackage{tikz}
\usetikzlibrary{arrows,calc}

\numberwithin{equation}{section}
\theoremstyle{plain}
\newtheorem{thm}{Theorem}[section]
\newtheorem{lem}[thm]{Lemma}
\newtheorem{prop}[thm]{Proposition}
\newtheorem{cor}[thm]{Corollary}

\newtheorem*{thm*}{Theorem}
\newtheorem*{lem*}{Lemma}
\newtheorem*{prop*}{Proposition}
\newtheorem*{cor*}{Corollary}
\theoremstyle{definition}
\newtheorem{defn}[thm]{Definition}
\newtheorem*{defn*}{Definition}

\newtheorem{ex}[thm]{Example}
{}
\newtheorem{rem}[thm]{Remark}
\newtheorem*{rem*}{Remark}
\newtheorem{setup}[thm]{Setup}
\newtheorem*{setup*}{Setup}

{}
{}
{}
\theoremstyle{remark}
{}
{}
{}

\def\iso{\cong}
\def\tensor{\otimes}
\def\a{\alpha}
\def\S{\Sigma}
\def\Z{\mathbb{Z}}
\def\D{\mathsf{D}}
\def\mcT{\mathcal{T}}
\def\mcU{\mathcal{U}}
\def\sfT{\mathsf{T}}
\def\sfU{\mathsf{U}}
\def\sfS{\mathsf{S}}
\def\mcS{\mathcal{S}}
\def\mcV{\mathcal{V}}
\def\sfV{\mathsf{V}}
\def\mcA{\mathcal{A}}
\def\mcB{\mathcal{B}}
\def\mcF{\mathcal{F}}
\def\sfF{\mathsf{F}}
\def\mcG{\mathcal{G}}

\def\sfM{\mathsf{M}}
\def\unit{\mathbf{1}}
\def\op{\mathsf{op}}

\DeclareMathOperator{\ob}{ob}

\DeclareMathOperator{\im}{im}
\DeclareMathOperator{\Hom}{Hom}

\DeclareMathOperator{\colim}{colim}
\DeclareMathOperator{\hocolim}{hocolim}
\DeclareMathOperator{\Modcat}{\mathsf{Mod}}
\DeclareMathOperator{\cModcat}{\mathcal M\mathrm{od}}
\DeclareMathOperator{\Ident}{Id}

\newcommand{\commd}[2]{\begin{tikzpicture}[scale=.75]#1\path[->,font=\scriptsize]#2;\end{tikzpicture}}
\newcommand{\xto}{\xrightarrow}

\renewcommand{\to}{\longrightarrow}

\title[Enriched representability theorems]{Enrichment and representability for triangulated categories}
\author{Johan Steen}
\address{Johan Steen \\ Institutt for matematiske fag \\ NTNU \\ 7491 Trondheim \\ Norway}
\email{johan.steen@math.ntnu.no}
\author{Greg Stevenson}
\address{Greg Stevenson \\ Universit\"at Bielefeld \\ Fakult\"at f\"ur Mathematik \\ BIREP Gruppe \\ Postfach 10\,01\,31 \\ 33501 Bielefeld \\ Germany}
\email{gstevens@math.uni-bielefeld.de}

\thanks{The first author was supported by a Norwegian Research Council project (NFR~231000) and the second author was partly supported by a fellowship from the Alexander von Humboldt Foundation during the period in which this research was conducted.}

\begin{document}

%\subjclass[2000]{?}

%\keywords{}

\begin{abstract}
    \noindent Given a fixed tensor triangulated category $\sfS$ we consider triangulated categories $\sfT$ together with an $\sfS$-enrichment which is compatible with the triangulated structure of $\sfT$. It is shown that, in this setting, an enriched analogue of Brown representability holds when both $\sfS$ and $\sfT$ are compactly generated. A natural class of examples of such enriched triangulated categories categories are module categories over separable monoids in $\sfS$. In this context we prove a version of the Eilenberg--Watts theorem for exact coproduct and copower preserving $\sfS$-functors, i.e., we show that any such functor between the module categories of separable monoids in $\sfS$ is given by tensoring with a bimodule.
\end{abstract}

\maketitle

\tableofcontents

\section{Introduction}
Over the last three decades, the importance and strength of compatible monoidal structures on triangulated categories has been continually highlighted. This is, for instance, exemplified in classification theorems of Devinatz, Hopkins, and Smith \cite{DevinatzHopkinsSmith}, Neeman \cite{NeeChro}, and Thomason \cite{Thomclass}, which describe various lattices of thick tensor ideals in terms of associated topological spaces. Of particular motivational relevance for us is the article of Thomason where it is shown that the thick tensor ideals in the category of perfect complexes over a reasonable scheme are classified by certain subsets of the topological space underlying the scheme. In fact, one can even recover the space if one knows the lattice of tensor ideals. More recently, Balmer \cite{BaSpec} has produced a very elegant framework into which these classifications fit, and shown that from the perfect complexes on a reasonable scheme, together with the left derived tensor product, one can actually reconstruct the scheme and not just the space. This is a very striking result; rephrasing slightly, it implies that from this data one can recover anything one could produce from the original scheme. In particular, one can get an enhancement of the derived category. This indicates that the existence of an exact monoidal structure somehow rigidifies the otherwise frequently rather floppy derived category.

It thus seems natural to ask exactly how much one can extract from the existence of an exact monoidal structure on a triangulated category or, more generally, from an action of such a category on another triangulated category. One way of formalising this setting is to consider enriched categories: given a rigidly compactly generated tensor triangulated category $\sfS$ and a well behaved action of $\sfS$ on a compactly generated triangulated category $\sfT$ one can produce an $\sfS$-enrichment of $\sfT$ which is compatible with the triangulated structures. The aim of this paper is to begin exploring this setting and to do some advertising by showing that the presence of such an enrichment actually allows one to prove some ``enhancement-flavoured'' statements.

The first main part of the paper deals with extending Brown representability to our enriched context. Classically, representability theorems have been important in algebraic topology, dating back to Brown's result on the representability of certain functors out of the homotopy category.  The study of representability of cohomological functors out of triangulated categories is more recent, but has been very fruitful, starting with the pioneering work of Bousfield \cite{BousfieldLSH} and flourishing with B\"okstedt--Neeman \cite{NeeHolim} and Neeman \cite{NeeGrot}. It is in this last paper that Neeman proves the Brown representability theorem for compactly generated triangulated categories, which is an immensely useful tool. The fundamental importance of Brown representability has led to it being generalised in related directions, see for instance \cite{ChornyBrown}, and prompted us to ask if one could adapt it to the enriched setting. It turns out that the answer is `yes' as we show in Theorem~\ref{thm_brown} --- Neeman's original proof is easily modified once one finds the correct hypotheses in the enriched setting.

\begin{thm}[{see Theorem~\ref{thm_brown}}]
    Let $\sfS$ be a compactly generated tensor triangulated category.  Assume that $\mcT$ is a copowered $\sfS$-category whose underlying category $\sfT$ is compactly generated triangulated.  Then any $\sfS$-functor $\mcF\colon \mcT^\op \to \mcS$ which preserves powers and has an underlying exact and coproduct preserving functor is representable.
\end{thm}

The precise conditions on the compatibility of the triangulated structures (given below) immediately shows that any representable functor is precisely of this form.

The second part of the paper deals with a tensor triangular version of the Eilenberg--Watts theorem. The classical Eilenberg--Watts theorem asserts that, given rings $A$ and $B$, any colimit preserving functor from $A$-modules to $B$-modules is given by tensoring with a bimodule, namely the image of $A$ under said functor. The proof of this theorem is rather elegant, and has been generalised to many other situations including to the setting of model categories by Hovey \cite{Hovey_EW}. The key abstract components are a suitable ambient category (abelian groups in the classical setting) in which to consider bimodule objects and an enrichment (again in abelian groups in the classical setting) in order to define the desired natural transformation.

There is, in general, no ambient triangulated category that could play the role of the category of abelian groups. However, given a separable monoid $A$ in a fixed tensor triangulated category $\sfS$, it has been shown by Balmer that the category of $A$-modules in $\sfS$ (in the naive sense) is triangulated \cite{Balmer_sep}. Thus, given two such monoids $A$ and $B$, we can consider tensor products and bimodules in the ambient category $\sfS$. Moreover, in this context, one can naturally view the module categories over $A$ and $B$ as enriched in $\sfS$, and we show that there is an analogue of the Eilenberg--Watts theorem for $\sfS$-functors between them.

\begin{thm}[see Theorem~{\ref{thm:tensorrep}}] \label{thm:informaltensorrep}
    Let $A$ and $B$ be separable monoids in a tensor triangulated category $\sfS$ compactly generated by the tensor unit $\unit$.  An $\sfS$-functor $\mcF\colon \cModcat_\sfS A \to \cModcat_\sfS B$ preserves copowers and has an exact and coproduct preserving underlying functor if and only if $\mcF \iso -\tensor_A Y$, for some $A$-$B$-bimodule $Y$.
\end{thm}

The paper is organized as follows: In Section 2 we give the necessary background on enriched categories for stating and proving the main theorems.  In Section 3 we prove our enriched analogue of Brown representability, namely that certain power preserving functors out of enriched categories are representable.  In Section 4 we recall, with significant detail, the relevant facts from the theory of separable monoids and prove our enriched Eilenberg--Watts theorem.

\section{Preliminaries on enriched categories}\label{sec:prelim}

Let $(\sfV, \otimes, \unit)$ be a closed symmetric monoidal category, whose internal hom we denote by $\mcV(-,-)$. We recall that a $\sfV$-category (or a category enriched in $\sfV$) $\mcA$ is a collection of objects $\ob \mcA$, for each $a,a'\in \ob \mcA$ an object of morphisms $\mcA(a,a')\in\sfV$, composition maps for each $a,a',a''\in \ob \mcA$
\begin{displaymath}
    \mcA(a',a'')\otimes \mcA(a,a') \stackrel{\circ}{\to} \mcA(a,a'')
\end{displaymath}
and units $i_a \colon \unit \to \mcA(a,a)$ such that the natural associativity and unitality constraints are satisfied.  The category $\sfV$ naturally gives rise to a $\sfV$-category $\mcV$ whose objects of morphisms for $x,y\in\sfV$ are $\mcV(x,y)$. For the precise diagrams that must be satisfied, further details on the self-enrichment of $\sfV$, and a more complete treatment of the facts we recall here the reader can consult \cite{Kelly}.

Given $\sfV$-categories $\mcA$ and $\mcB$ a \emph{$\sfV$-functor} $\mcF\colon \mcA \to \mcB$ is given by an assignment, which is also denoted by $\mcF$
\begin{displaymath}
\mcF\colon \ob \mcA \to \ob\mcB
\end{displaymath}
together with maps in $\sfV$ for all $a,a' \in \mcA$
\begin{displaymath}
    \mcF_{a,a'}\colon \mcA(a,a') \to \mcB(\mcF a, \mcF a').
\end{displaymath}
These maps must be compatible with composition in the sense that the diagrams
\begin{displaymath}
    \commd{
        \node (1) at (0,2) {$\mcA(a',a'')\otimes \mcA(a,a')$};
        \node (2) at (8,2) {$\mcA(a,a'')$};
        \node (3) at (0,0) {$\mcB(\mcF a', \mcF a'') \otimes \mcB(\mcF a, \mcF a')$};
        \node (4) at (8,0) {$\mcB(\mcF a, \mcF a'')$};
    }{
        (1) edge node[auto] {$\circ_\mcA$} (2)
        (3) edge node[auto] {$\circ_\mcB$} (4)
        (1) edge node[auto] {$\mcF_{a',a''} \otimes \mcF_{a,a'}$} (3)
        (2) edge node[auto] {$\mcF_{a,a''}$} (4)
    }
\end{displaymath}
commute for all $a,a',a'' \in \mcA$. They must also be unital, i.e., for all $a\in \mcA$ the triangle
\begin{displaymath}
    \commd{
        \node (1) at (0,2) {$\unit$};
        \node (2) at (4,2) {$\mcA(a,a)$};
        \node (3) at (4,0) {$\mcB(\mcF a,\mcF a)$};
    }{
        (1) edge node[auto] {$i_a$} (2)
        (1) edge node[auto,swap] {$i_{\mcF_a}$} (3)
        (2) edge node[auto] {$\mcF_{a,a}$} (3)
    }
\end{displaymath}
commutes. Suppose $\mcG\colon \mcA \to \mcB$ is an additional $\sfV$-functor. A \emph{$\sfV$-natural transformation} $\a\colon \mcF\to \mcG$ is given by components
\begin{displaymath}
    \a_a\colon \unit \to \mcB(\mcF a, \mcG a),\quad a\in\mcA
\end{displaymath}
such that the following hexagons, which express the naturality of $\a$, commute for all $a,a'\in \mcA$:
\begin{displaymath}
    \commd{
        \node (1) at (2,0) {$\mcA(a,a')\otimes \unit$};
        \node (2) at (8,0) {$\mcB(\mcG a,\mcG a') \otimes \mcB(\mcF a,\mcG a)$};
        \node (3) at (0,2) {$\mcA(a,a')$};
        \node (4) at (10,2) {$\mcB(\mcF a, \mcG a')$};
        \node (5) at (2,4) {$\unit\otimes\mcA(a,a')$};
        \node (6) at (8,4) {$\mcB(\mcF a',\mcG a')\otimes \mcB(\mcF a,\mcF a')$};
    }{
        (3) edge node[auto] {} (1)
        (3) edge node[auto] {} (5)
        (1) edge node[auto,swap] {$\mcG_{a,a'}\otimes \a_a$} (2)
        (5) edge node[auto] {$\a_{a'} \otimes \mcF_{a,a'}$} (6)
        (2) edge node[auto,swap] {$\circ$} (4)
        (6) edge node[auto] {$\circ$} (4)
    }
\end{displaymath}

Let $\mcA$ be a $\sfV$-category. The \emph{underlying category} $\mcA_0$ of $\mcA$ is the usual category with objects $\ob \mcA$ and
\begin{displaymath}
    \mcA_0(a,a') = \sfV\big(\unit, \mcA(a,a')\big).
\end{displaymath}
The composition and units in $\mcA_0$ are induced from $\mcA$ in the obvious way. This construction defines a $2$-functor from $\sfV$-categories to categories: given a $\sfV$-functor $\mcF\colon \mcA \to \mcB$, its underlying functor $\mcF_0\colon \mcA_0 \to \mcB_0$ has the same assignment on objects, its action on maps is given by $\sfV(\unit, \mcF_{-,-})$ and the natural transformations essentially do not change.

Given $a\in \mcA$ and $v\in \sfV$ the \emph{copower} of $a$ by $v$, if it exists, is an object $v\odot a$ of $\mcA$ together with natural isomorphisms in $\sfV$
\begin{displaymath}
    \mcA(v\odot a, a') \cong \mcV\big(v, \mcA(a,a')\big)
\end{displaymath}
for all $a'\in \mcA$. Dually the \emph{power} of $a$ by $v$, if it exists, is an object $v\pitchfork a$ of $\mcA$ together with natural isomorphisms in $\sfV$
\begin{displaymath}
    \mcA(a', v \pitchfork a) \cong \mcV\big(v, \mcA(a',a)\big)
\end{displaymath}
for all $a'\in \mcA$.

If all (co)powers exist we say that $\mcA$ is a \emph{(co)powered} $\sfV$-category.

\begin{ex}
    One sees easily from the definition that all copowers and powers exist in $\mcV$. Indeed, one has the equalities for $x,y \in \mcV$
    \begin{displaymath}
        x\odot y = x\otimes y \quad \text{and} \quad x\pitchfork y = \mcV(x,y);
    \end{displaymath}
    the defining isomorphisms for (co)powers express the adjunction between $\otimes$ and $\mcV(-,-)$.
\end{ex}

We will not require much technology concerning powers and copowers. However, we will need the following, rather standard, lemma.

\begin{lem} \label{lem_power_pres}
    Let $\mcA$ and $\mcB$ be powered $\sfV$-categories and let $\mcF\colon \mcA \to \mcB$ be a $\sfV$-functor. Given $v\in \sfV$ and $a\in \mcA$ there is a natural map
    \begin{displaymath}
        \mcF(v\pitchfork a) \to v \pitchfork \mcF a.
    \end{displaymath}
\end{lem}
\begin{proof}
    The desired morphism is given by following the identity map through the following diagram:
    \begin{displaymath}
        \commd{
            \node (1) at (0,2) {$\mcA(v \pitchfork a, v\pitchfork a)$};
            \node (2) at (8,2) {$\mcB\big(\mcF(v \pitchfork a), v \pitchfork \mcF a\big)$};
            \node (3) at (0,0) {$\mcV\big(v,\mcA(v \pitchfork a,a)\big)$};
            \node (4) at (8,0) {$\mcV\big(v,\mcB(\mcF(v \pitchfork a), \mcF a)\big)$};
        }{
            (1) edge node[auto] {$\wr$} (3)
            (4) edge node[auto] {$\wr$} (2)
            (3) edge node[auto] {$\mcV(v, \mcF_{v \pitchfork a,a})$} (4)
        }\qedhere
    \end{displaymath}
\end{proof}

Using the lemma we can make sense of the statement that the functor $\mcF$ \emph{preserves powers}, i.e., that for all $v\in \sfV$ and $a\in \mcA$ the  morphism of the lemma is an isomorphism.  Of course there is the following dual statement which we shall also use.
\begin{lem}\label{lem_copower_pres}
    Let $\mcA$ and $\mcB$ be copowered $\sfV$-categories and let $\mcF\colon \mcA \to \mcB$ be a $\sfV$-functor. Given $v\in \sfV$ and $a\in \mcA$ there is a natural map
    \begin{displaymath}
        v\odot \mcF a \to \mcF(v\odot a).
    \end{displaymath}
\end{lem}

\section{Enriched Brown representability}
This section is devoted to the first of our main results, namely that Brown representability holds in the enriched setting. Let us begin by introducing the players and formulating what we mean by enriched Brown representability.

\begin{setup} \label{setup}
    Let $\sfS$ be a compactly generated tensor triangulated category, i.e., $\sfS$ is a compactly generated triangulated category with a closed symmetric monoidal structure $(\otimes, \mcS(-,-), \unit)$ such that $\otimes$ is exact in both variables. We moreover assume that the internal hom $\mcS(-,-)$ is exact in both variables. We assume throughout that the compact objects of $\sfS$ form a tensor subcategory, i.e., the unit $\unit$ is compact and the tensor product of two compacts is compact. Following our earlier conventions we denote by $\mcS$ the self-enrichment of $\sfS$.
		
		We fix an $\sfS$-category $\mcT$ with copowers (i.e., $\mcT^\op$ has powers) such that the underlying category, denoted $\sfT$, carries the structure of a compactly generated triangulated category (which is also fixed throughout this section). We will assume that the triangulated structure of $\sfT$ is compatible with the $\sfS$-enrichment in the sense that the functors
    \begin{displaymath}
        \mcT(t,-)_0 \quad \text{and} \quad \mcT(-,t)_0
    \end{displaymath}
    underlying the $\sfS$-functors corepresented and represented by $t\in \mcT$, are exact for all $t$. We also require that for a compact object $c\in \sfS$ and a compact object $t\in \sfT$ the copower $c\odot t$ is again compact in $\sfT$.
\end{setup}

\begin{rem}
    One source, at least morally, of such $\mcT$ is the theory of actions of compactly generated tensor triangulated categories. Given $\sfS$ as above with a sufficiently nice action, in the sense of \cite{StevensonActions}, on a compactly generated triangulated category $\sfT$ one obtains an $\sfS$-category $\mcT$ whose underlying category is canonically identified with $\sfT$. Further details concerning this intuition can be found in \cite{KellyJanelidze} and also in Section~\ref{sec_EW} (see in particular Lemma~\ref{lem:action-enr} and Proposition~\ref{prop:action-enr-fun}).
\end{rem}

\begin{defn}
    We say that $\mcT$ satisfies \emph{enriched Brown representability} if every power preserving functor $\sfS$-functor $\mcF\colon \mcT^\op \to \mcS$, such that the underlying functor $\sfF = \mcF_0$ is exact and preserves products, is isomorphic to a representable $\sfS$-functor.
\end{defn}

\begin{rem}
    Since products in $\sfT^\op$ are precisely the coproducts in $\sfT$, the phrase ``$\sfF$ preserves products'' means that it sends coproducts to products, which is precisely the assumption in the usual Brown representability theorem. Similarly, the power preservation hypothesis can be unwound as saying $\mcF$ sends copowers to powers.
\end{rem}

We shall prove that, given $\sfS$ and $\mcT$ as in Setup \ref{setup}, the category $\mcT$ satisfies enriched Brown representability. Our argument parallels Neeman's proof \cite[Theorem~3.1]{NeeGrot} of the usual Brown representability theorem for compactly generated triangulated categories. The only real adaptation required is to avoid using morphisms in $\mcT$ (as there is not necessarily such a notion), and this is fairly standard. The most important observation is, in some sense, that the correct condition for enriched representability is not just that $\sfF$ should commute with products but that $\mcF$ also should preserve powers; this is not visible when the tensor unit generates $\sfS$, but is crucial for our method of extending from the case that $\sfS$ is generated by $\unit$ to the general case.

Let $\sfS$ and $\mcT$ be as in Setup~\ref{setup} and fix an $\sfS$-functor $\mcF\colon \mcT^\op \to \mcS$ whose underlying functor we denote by $\sfF$. As indicated above we shall assume that $\mcF$ preserves powers and that $\sfF$ is exact and commutes with products. As in Neeman's proof we begin by constructing a tower of objects in $\mcT$ whose corresponding representable functors approximate $\mcF$.

Let $G$ be a suspension closed compact generating set for $\sfT$, for example one could take $G$ to be a skeleton for the compacts $\sfT^c$. Set
\begin{displaymath}
    U_0 = \{(g,f) \; \vert \; g\in G, \; f\colon \unit \to \mcF g\}
\end{displaymath}
and form the corresponding coproduct
\begin{displaymath}
    X_0 = \coprod_{(g,f)\in U_0} g
\end{displaymath}
in $\mcT$. By the weak form of the enriched Yoneda lemma we have isomorphisms
\begin{align*}
    \Hom\big(\mcT(-,X_0), \mcF\big) &\cong \sfS(\unit, \mcF X_0) \\
        &\cong \sfS\big(\unit, \mcF(\coprod_{(g,f)\in U_0} g)\big) \\
        &\cong \sfS(\unit, \prod_{(g,f)\in U_0} \mcF g) \\
        &\cong \prod_{(g,f)\in U_0} \sfS(\unit, \mcF g)
\end{align*}
and so
\begin{displaymath}
    \prod_{(g,f)\in U_0} f \in \prod_{(g,f)\in U_0} \sfS(\unit, \mcF g)
\end{displaymath}
gives a canonical enriched natural transformation $\phi^0 \colon \mcT(-, X_0) \to \mcF$.

We now assume, inductively, that we have constructed objects $X_i \in \mcT$ together with morphisms
\begin{displaymath}
    \psi^i \in \sfT(X_{i-1}, X_i) \quad \text{and} \quad \phi^i \in \Hom\big(\mcT(-,X_i), \mcF\big)
\end{displaymath}
such that the triangles
\begin{displaymath}
    \commd{
        \node (1) at (0,2) {$\mcT(-,X_{i-1})$};
        \node (2) at (3,0) {$\mcF$};
        \node (3) at (6,2) {$\mcT(-, X_i)$};
    }{
        (1) edge node[auto,swap] {$\phi^{i-1}$} (2)
        (1) edge node[auto] {$\mcT(-, \psi^i)$} (3)
        (3) edge node[auto] {$\phi^i$} (2)
    }
\end{displaymath}
commute. The object $X_{i+1}$ and maps $\phi^{i+1}$ and $\psi^{i+1}$ are constructed as follows: set
\begin{displaymath}
    U_{i+1} = \coprod_{g\in G} \ker \sfS\big(\unit, \mcT(g, X_i)\big) \xto{\sfS(\unit, \phi^i_g)} \sfS(\unit, \mcF g)
\end{displaymath}
and consider the coproduct
\begin{displaymath}
    K_{i+1} = \coprod_{(g,f) \in U_{i+1}} g,
\end{displaymath}
where $(g,f) \in U_{i+1}$ is our notation for the morphism
\begin{displaymath}
    f\in \sfS\big(\unit, \mcT(g, X_i)\big) = \sfT(g, X_i)
\end{displaymath}
occurring in $U_{i+1}$. There is a canonical morphism $K_{i+1} \to X_i$ in $\sfT$ and we complete it to a triangle
\begin{displaymath}
    K_{i+1} \to X_i \xto{\psi^{i+1}} X_{i+1} \to \S K_{i+1}
\end{displaymath}
defining $X_{i+1}$.
We now produce the morphism $\phi^{i+1}$. Applying the exact functor $\sfF$ gives a triangle in $\sfS$
\begin{displaymath}
    \prod_{(g,f)\in U_{i+1}}\mcF g \cong \mcF K_{i+1} \longleftarrow \mcF X_i \xleftarrow{\sfF\psi^{i+1}} \mcF X_{i+1} \longleftarrow \S^{-1}\mcF K_{i+1}.
\end{displaymath}
By the Yoneda lemma the map $\phi^i$ corresponds to a morphism, which we also call $\phi^i$, $\unit \to \mcF X_i$. We claim that the latter map factors via $\mcF X_{i+1}$ giving, by Yoneda, the desired natural transformation $\phi^{i+1}$.

By construction $\mcF X_i \to \mcF K_{i+1}$ corresponds to the components
\begin{displaymath}
    \unit \stackrel{f}{\to} \mcT(g, X_i) \xto{\mcF_{g, X_i}} \mcS(\mcF X_i, \mcF g).
\end{displaymath}
Applying the $\otimes$-$\mcS(-,-)$ adjunction
\begin{displaymath}
    \sfS\big(\mcT(g, X_i), \mcS(\mcF X_i, \mcF g)\big) \cong \sfS\big(\mcF X_i \otimes \mcT(g, X_i), \mcF g\big)
\end{displaymath}
to
\begin{displaymath}
    \mcF_{g, X_i}\colon \mcT(g, X_i) \to \mcS(\mcF X_i, \mcF g),
\end{displaymath}
the component at $(g,f)$ of the composite $\unit \stackrel{\phi^i}{\to} \mcF X_i \to \mcF K_{i+1}$ can be written as the composite
\begin{displaymath}
    \unit \stackrel{\sim}{\to} \unit\otimes \unit \xto{\phi^i \otimes f} \mcF X_i \otimes \mcT(g, X_i) \to \mcF g.
\end{displaymath}
On the other hand, the above composite can be identified with
\begin{displaymath}
    \unit \stackrel{f}{\to} \mcT(g, X_i) \stackrel{\phi^i_g}{\to} \mcF g,
\end{displaymath}
which is zero by construction as $f\in \ker \sfS\big(\unit, \mcT(g, X_i)\big) \xto{\sfS(\unit, \phi^i_g)} \sfS(\unit, \mcF g)$. This shows
\begin{displaymath}
    \unit \stackrel{\phi^i}{\to} \mcF X_i \to \mcF K_{i+1}
\end{displaymath}
is zero in $\sfS$ and thus $\phi^i$ can be factored via a morphism $\phi^{i+1}\colon \unit \to \mcF X_{i+1}$. Equivalently, we have a natural transformation of enriched functors, which we also denote by $\phi^{i+1}$, making the following triangle commute
\begin{displaymath}
    \commd{
        \node (1) at (0,2) {$\mcT(-,X_i)$};
        \node (2) at (3,0) {$\mcF$.};
        \node (3) at (6,2) {$\mcT(-, X_{i+1})$};
    }{
        (1) edge node[auto,swap] {$\phi^i$} (2)
        (1) edge node[auto] {$\mcT(-, \psi^{i+1})$} (3)
        (3) edge node[auto] {$\phi^{i+1}$} (2)
    }
\end{displaymath}
Indeed, this triangle commutes by construction since the triangle
\begin{displaymath}
    \commd{
        \node (1) at (0,2) {$\mcF X_i$};
        \node (2) at (3,0) {$\unit$};
        \node (3) at (6,2) {$\mcF X_{i+1}$};
    }{
        (2) edge node[auto] {$\phi^i$} (1)
        (3) edge node[auto,swap] {$\sfF\psi^{i+1}$} (1)
        (2) edge node[auto,swap] {$\phi^{i+1}$} (3)
    }
\end{displaymath}
commutes in $\sfS$.

We now define an object $X$ of $\mcT$ by
\begin{displaymath}
X = \hocolim_i X_i,
\end{displaymath}
i.e., by the triangle in $\sfT$
\begin{displaymath}
    \coprod_i X_i \xto{1-\psi^i} \coprod_i X_i \to X \to \S \coprod_i X_i.
\end{displaymath}
Applying $\sfF$ to this defining triangle yields a triangle in $\sfS$
\begin{displaymath}
    \prod_i \mcF X_i \xleftarrow{1- \sfF\psi^{i+1}} \prod_i \mcF X_i \longleftarrow \mcF X \longleftarrow \S^{-1} \prod_i \mcF X_i.
\end{displaymath}
By the compatibility conditions between the $\phi^i$ and $\psi^i$, the composite
\begin{displaymath}
    \unit \to \prod_i \mcF X_i \xto{1-\sfF\psi^i} \prod_i \mcF X_i,
\end{displaymath}
where the first morphism is induced by the $\phi^i$, vanishes and so the triangle gives us a factorization of $\unit \to \prod_i \mcF X_i$ via $\phi \colon \unit \to \mcF X$. This map $\phi$ is compatible with the $\phi^i$ and $\psi^i$ in the obvious way.

Corresponding to $\phi$ we have an enriched natural transformation
\begin{displaymath}
    \phi\colon \mcT(-,X) \to \mcF.
\end{displaymath}
We will prove that $\phi$ is an isomorphism of $\sfS$-functors, i.e., each of the components
\begin{displaymath}
    \phi_Y\colon \unit \to \mcS\big(\mcT(Y,X), \mcF Y\big)
\end{displaymath}
or, more precisely, the maps they correspond to
\begin{displaymath}
	\phi_Y\colon \mcT(Y,X) \to \mcF Y
\end{displaymath}
are isomorphisms in $\sfS$. First we observe that it is enough to check this on generators.

\begin{lem}\label{lem:brownloc}
The full subcategory
\begin{displaymath}
    \sfM = \{Y\in \sfT \; \vert\; \phi_Y \; \text{is an isomorphism}\}
\end{displaymath}
is localizing in $\sfT$. In particular, if $G\subseteq \sfM$ then $\sfM = \sfT$ and so $\phi$ is an isomorphism.
\end{lem}
\begin{proof}
The underlying natural transformation of $\phi$, whose components are just the $\phi_Y$, is a natural transformation between the exact product preserving functors $\sfT(-,X)$ and $\sfF$ (recall product preservation here means sending coproducts to products). The usual argument shows $\sfM$ is localizing: the suspension of an isomorphism is an isomorphism, as is any product of isomorphisms, and any completion of two isomorphisms to a morphism of triangles.

The final statement is then clear, as any localizing subcategory of $\sfT$ containing $G$ must be $\sfT$ itself. By definition of $\sfM$ this says that $\phi_Y$ is an isomorphism for all $Y\in \sfT$, i.e., $\phi$ is a natural isomorphism.
\end{proof}

Our strategy to check that the $\phi_g$ are isomorphisms for $g\in G$ is as follows. We can complete $\phi_g$ to a triangle in $\sfS$
\begin{displaymath}
    \mcT(g, X) \stackrel{\phi_g}{\to} \mcF g \to Z_g \to \S \mcT(g,X)
\end{displaymath}
and it is sufficient to show $Z_g \cong 0$. The first step in proving that $Z_g$ vanishes is the following lemma.

\begin{lem}\label{lem:nomaps}
There are no morphisms from $\unit$ to $\S^i Z_g$ for any $i\in \Z$, i.e.,
\begin{displaymath}
    \sfS(\unit, \S^i Z_g) = 0 \quad \forall \; i\in \Z.
\end{displaymath}
\end{lem}
\begin{proof}
Applying $\sfS(\unit, -)$ to the triangle defining $Z_g$ gives a commutative diagram
\begin{displaymath}
    \commd{
        \node (1) at (0,2) {$\sfS(\unit, \S^{i-1} Z_g)$};
        \node (2) at (4,2) {$\sfS\big(\unit, \S^i\mcT(g, X)\big)$};
        \node (3) at (8,2) {$\sfS(\unit, \S^i\mcF g)$};
        \node (4) at (12,2) {$\sfS(\unit, \S^{i}Z_g)$};
        \node (5) at (4,0) {$\sfS\big(\unit, \mcT(\S^{-i}g, X)\big)$};
        \node (6) at (8,0) {$\sfS\big(\unit, \mcF(\S^{-i}g)\big)$};
    }{
        (1) edge node[auto] {} (2)
        (2) edge node[auto] {} (3)
        (3) edge node[auto] {} (4)
        (2) edge node[auto] {$\wr$} (5)
        (3) edge node[auto] {$\wr$} (6)
        (5) edge node[auto] {} (6)
    }
\end{displaymath}
where the top row is exact and $\S^{-i}g\in G$ by the assumption that $G$ is suspension closed. Suspension closure of $G$ together with the above diagram means it is sufficient to consider the case $i=0$, i.e., show that
\begin{displaymath}
    \sfS\big(\unit, \mcT(g,X)\big) \to \sfS(\unit,\mcF g)
\end{displaymath}
is an isomorphism.  We now use the identifications
\begin{align*}
    \sfS(\unit, \mcT(g, X))
        &= \sfT(g, X) \\
        &= \sfT(g, \hocolim_i X_i) \\
        &\cong \colim_i \sfT(g, X_i).
\end{align*}
By construction $\sfT(g, X_0) \to \sfS(\unit, \mcF g)$ is surjective and fits into the commutative triangle
\begin{displaymath}
    \commd{
        \node (1) at (0,2) {$\sfT(g, X_0)$};
        \node (2) at (4,2) {$\sfS(\unit, \mcF g)$};
        \node (3) at (2,0) {$\sfT(g, X)$};
    }{
        (1) edge node[auto] {} (2)
        (1) edge node[auto] {} (3)
        (3) edge node[auto] {} (2)
    }
\end{displaymath}
showing that $\sfT(g, X) \to \sfS(\unit, \mcF g)$ is also surjective.

Now we prove that it is also injective. Suppose we are given
\begin{displaymath}
    f\in \ker \big(\colim_i \sfT(g, X_i) \to \sfS(\unit, \mcF g)\big).
\end{displaymath}
It can be represented by some $f_i \in \sfT(g, X_i)$ which is then necessarily in the kernel of the composite
\begin{displaymath}
    \sfS\big(\unit, \mcT(g, X_i)\big) = \sfT(g, X_i) \to \colim_i \sfT(g, X)\to \sfS(\unit, \mcF g).
\end{displaymath}
Hence $f_i$ is an element of the set $U_{i+1}$ which we used in defining $X_{i+1}$. Commutativity of
\begin{displaymath}
    \commd{
        \node (1) at (0,2) {$\sfT(g,X_i)$};
        \node (2) at (3,0) {$\sfT(g,X)$};
        \node (3) at (6,2) {$\mcT(g,X_{i+1})$};
    }{
        (1) edge node[auto] {} (2)
        (1) edge node[auto] {$\sfT(g, \psi^{i+1})$} (3)
        (3) edge node[auto] {} (2)
    }
\end{displaymath}
then implies, by the way $\psi^{i+1}$ was defined, that the image of $f_i$ in $\sfT(g, X)$, which is none other than $f$, is zero. Thus the map $\sfT(g,X) \to \sfS(\unit, \mcF g)$ is injective.

So we have proved that $\sfT(g,X) \to \sfS(\unit, \mcF g)$ is an isomorphism for any $g\in G$. The exact sequence considered at the beginning of the proof then forces $\sfS(\unit, Z_g)$ to be zero as claimed.
\end{proof}

\begin{rem}
This is already enough to prove representability in the case $\sfS$ is generated by $\unit$.
\end{rem}

\begin{thm}\label{thm_brown}
    Let $\sfS$ be a compactly generated tensor triangulated category whose compact objects are a tensor subcategory and denote by $\mcS$ the self-enrichment of $\sfS$. Let $\mcT$ be a copowered $\sfS$-category whose underlying category $\sfT$ carries a fixed structure of compactly generated triangulated category. Finally, suppose the operation of copowering by an object of $\sfS^c$ sends compacts to compacts in $\sfT$. Then any power preserving $\sfS$-functor $\mcF \colon \mcT^\op\to \mcS$ whose underlying functor is exact and preserves products is representable, i.e., there is an $X\in \mcT$ with
    \begin{displaymath}
        \mcT(-, X) \cong \mcF.
    \end{displaymath}
\end{thm}
\begin{proof}
    We shall prove that the map $\phi\colon \mcT(-, X)\to \mcF$ which we constructed earlier is an isomorphism.  By Lemma~\ref{lem:brownloc} it is enough to check this on our compact generating set $G$ for $\sfT$. Let $c$ be a compact object of $\sfS$, and let $g\in G$.  Both $\mcF$ and $\mcT(-,X)$ preserve powers so we get a diagram
    \begin{displaymath}
        \commd{
            \node (1) at (0,4) {$\sfS\big(c, \mcT(g, X)\big)$};
            \node (2) at (5,4) {$\sfS(c, \mcF g)$};
            \node (3) at (0,2) {$\sfS\big(\unit, \mcS\big(c, \mcT(g, X)\big)\big)$};
            \node (4) at (5,2) {$\sfS\big(\unit, \mcS(c, \mcF g)\big)$};
            \node (5) at (0,0) {$\sfS\big(\unit, \mcT(c\odot g, X)\big)$};
            \node (6) at (5,0) {$\sfS\big(\unit, \mcF(c\odot g)\big)$};
        }{
            (1) edge node[auto] {} (2)
            (3) edge node[auto] {} (4)
            (5) edge node[auto] {} (6)
            (1) edge node[auto] {$\wr$} (3)
            (3) edge node[auto] {$\wr$} (5)
            (2) edge node[auto] {$\wr$} (4)
            (4) edge node[auto] {$\wr$} (6)
}
\end{displaymath}
which commutes by naturality. Without loss of generality we may assume our generating set $G$ is closed under copowering with objects of $\sfS^c$. Thus Lemma~\ref{lem:nomaps} applies to show, by considering the triangle in $\sfS$
\begin{displaymath}
    \mcT(c\odot g, X) \to \mcF(c\odot g) \to Z_{c\odot g} \to \S\mcT(c\odot g, X),
\end{displaymath}
that the bottom morphism in the above diagram is an isomorphism. Hence for any compact object $c\in \sfS^c$
\begin{displaymath}
    \sfS\big(c, \mcT(g,X)\big)\to \sfS(c, \mcF g)
\end{displaymath}
is an isomorphism. Compact generation of $\sfS$ then implies $\phi_g\colon \mcT(g,X) \to \mcF g$ is an isomorphism. Since $g\in G$ was arbitrary this completes the proof.
\end{proof}

\section{Triangulated module categories}\label{sec_EW}
We now turn to the question of representing covariant functors via bimodule objects. This is a more delicate question as, in general, an abstract triangulated category is not some subcategory of objects with extra structure in a ``universal ambient'' triangulated category where such bimodules could exist. However, we are still able to prove representability results for functors between certain triangulated categories. Let us begin by fixing the setup and some conventions.

Throughout $\sfS$ will denote a compactly generated tensor triangulated category and $\mcS$ will denote $\sfS$ considered as a category enriched in itself. We will always make the assumption that $\sfS$ is generated by the tensor unit, i.e.,
\begin{displaymath}
    \sfS = \langle \unit \rangle.
\end{displaymath}
Finally, we will assume that $\sfS$ is $\infty$-triangulated in the sense that one has higher octahedra and the corresponding compatibility axioms for them. Let us allay any potential worry this last sentence could have caused by pointing out right away that we shall not explicitly deal with this higher structure. It is a technical assumption required in the work of Balmer \cite{Balmer_sep} which forms the basis for our results. The reader who desires further details should consult the work of K\"{u}nzer \cite{kunzer} and Maltsiniotis \cite{maltsiniotis}; a compact presentation of the axioms can also be found in \cite{Balmer_sep}.

As $\sfS$ is monoidal one can consider monoid objects in $\sfS$.  We briefly recall that a monoid consist of an object, say, $A$, a multiplication $\mu\colon A\tensor A\to A$ and a unit $\eta\colon \unit \to A$ subject to the usual associativity and unitality diagrams.

Given such a monoid $A$, we define the category of \emph{right $A$-modules in $\sfS$}, denoted by $\Modcat_\sfS A$, to have as objects pairs $(x,\rho)$, where $\rho\colon x\tensor A\to x$ is compatible with the monoid structure in the natural way, i.e., the following two diagrams commute:
\begin{equation}
\begin{split} \label{eq:module}
\commd{
    \node (1) at (0,2) {$x\tensor A\tensor A$};
    \node (2) at (4,2) {$x\tensor A$};
    \node (3) at (0,0) {$x\tensor A$};
    \node (4) at (4,0) {$x$};
    \node (5) at (8,2) {$x\tensor\unit$};
    \node (6) at (12,2) {$x\tensor A$};
    \node (7) at (12,0) {$x$};
}{
    (1) edge node[auto] {$\rho\tensor1$} (2)
    (1) edge node[auto] {$1\tensor\mu$} (3)
    (2) edge node[auto] {$\rho$} (4)
    (3) edge node[auto] {$\rho$} (4)
    (5) edge node[auto] {$1\tensor\eta$} (6)
    (6) edge node[auto] {$\rho$} (7)
    (5) edge node[auto,swap] {$\iso$} (7)
}
\end{split}
\end{equation}
A morphism $f\colon (x,\rho_x) \to (y,\rho_y)$ is merely an $A$-linear morphism, i.e., a morphism $f\colon x\to y$ such that
\begin{displaymath}
\commd{
    \node (1) at (0,2) {$x\tensor A$};
    \node (2) at (4,2) {$y\tensor A$};
    \node (3) at (0,0) {$x$};
    \node (4) at (4,0) {$y$};
}{
    (1) edge node[auto] {$f\tensor1$} (2)
    (1) edge node[auto] {$\rho_x$} (3)
    (2) edge node[auto] {$\rho_y$} (4)
    (3) edge node[auto] {$f$} (4)
}
\end{displaymath}
commutes.  Note in particular that $(A,\mu)$ is a right module.  One defines the category of left $A$-modules and, given another monoid $B$, the category of $A$-$B$-bimodules similarly.

As $\sfS$ is \emph{symmetric} monoidal, there are isomorphisms $c_{x,y}\colon x\tensor y \stackrel\sim\to y\tensor x$, natural in both $x$ and $y$.  Thus any monoid $A$ admits another product, namely $\mu\circ c_{A,A}$.  We shall denote this opposite monoid by $A^\op$.  This allows us to view a left $A$-module in $\sfS$ as an object of $\Modcat_\sfS A^\op$ and an $A$-$B$-bimodule as an object of $\Modcat_\sfS A^\op\tensor B$.

\begin{defn}
    A monoid $A$ with multiplication $\mu$ is \emph{separable} if the multiplication map $\mu$ admits a bimodule section, i.e., a morphism $\sigma\colon A\to A\tensor A$ such that $\mu\sigma = 1$ and the following diagram commutes:
    \begin{displaymath}
    \commd{
        \node (1) at (4,4) {$A\tensor A$};
        \node (2) at (0,2) {$A\tensor A\tensor A$};
        \node (3) at (4,2) {$A$};
        \node (4) at (8,2) {$A\tensor A\tensor A$};
        \node (5) at (4,0) {$A\tensor A$};
    }{
        (1) edge node[auto,swap] {$\sigma\tensor1$} (2)
        (1) edge node[auto] {$\mu$} (3)
        (1) edge node[auto] {$1\tensor\sigma$} (4)
        (2) edge node[auto,swap] {$1\tensor\mu$} (5)
        (3) edge node[auto] {$\sigma$} (5)
        (4) edge node[auto] {$\mu\tensor1$} (5)
    }
    \end{displaymath}
\end{defn}
\begin{rem}
    A monoid $A$ gives rise to the extension of scalars functor
    \begin{displaymath}
        F_A = -\tensor A\colon \sfS \to \Modcat_\sfS A,
    \end{displaymath}
    which admits as a right adjoint the forgetful functor $U_A\colon \Modcat_\sfS A \to \sfS$.  By \cite[Prop.~3.11]{Balmer_sep} a monoid is separable if and only if $U_A$ is separable as a functor, i.e., the counit $\varepsilon_A\colon F_AU_A \to \Ident_{\Modcat_\sfS A}$ admits a section.
\end{rem}

We recall the following theorem due to Balmer, showing that the category of $A$-modules inherits a triangulated structure from $\sfS$ provided $A$ is separable.
\begin{thm}[{\cite[Cor.~5.18]{Balmer_sep}}]
    Let $\sfS$ be tensor $\infty$-triangulated and let $A\in\sfS$ be a separable monoid.  Then $\Modcat_\sfS A$ has a unique $\infty$-triangulation such that an $n$-triangle in $\Modcat_\sfS A$ is distinguished if and only if its image under $U_A$ is distinguished in $\sfS$.
\end{thm}
In particular $\Modcat_\sfS A$ is triangulated, and a triangle
\begin{displaymath}
    (x,\rho_x) \to (y,\rho_y) \to (z,\rho_z) \to (\S x, \S\rho_{x})
\end{displaymath}
is distinguished precisely when $x\to y\to z\to \S x$ is.

By the exactness of $\tensor$ on $\sfS$, this observation yields an action of $\sfS$ on $\Modcat_\sfS A$ in the sense of \cite{StevensonActions}. Indeed, Balmer's description of the triangulated structure immediately implies that for $s\in \sfS$ the functor
\begin{displaymath}
s\otimes - \colon \Modcat_\sfS A \to \Modcat_\sfS A
\end{displaymath}
sends distinguished triangles to distinguished triangles (since $s\otimes-$ is exact as an endofunctor of $\sfS$).

We now sketch that such an action gives rise to an enrichment. Note that this observation is certainly not new, and details can be found for instance in \cite{KellyJanelidze}. 
\begin{lem} \label{lem:action-enr}
    Let $\sfT$ be a triangulated category which admits an $\sfS$-action 
		\begin{displaymath}
		\ast\colon \sfS \times \sfT \to \sfT
		\end{displaymath}
		such that $\ast$ is exact and coproduct preserving in each variable. Then $\sfT$ admits an enrichment $\mcT$, in $\sfS$, such that $\mcT_0 = \sfT$. Moreover, $\mcT$ is copowered over $\sfS$.
\end{lem}
\begin{proof}
    We only give a sketch of the proof to fix ideas. Further details can be found in \cite{KellyJanelidze}.

    Fix $x\in\sfT$ and consider the functor $-\ast x\colon \sfS \to \sfT$.  It is exact and commutes with coproducts, and thus, by Brown representability, admits a right adjoint which we denote
    \begin{displaymath}
        \mcT(x,-)\colon \sfT \to \sfS.
    \end{displaymath}
    We claim that evaluating this functor at $y\in \sfT$ gives the hom object $\mcT(x,y)$ of a category $\mcT$ enriched in $\sfS$.

    First note that there is a natural evaluation morphism $\varepsilon_{x,y}\colon \mcT(x,y) \ast x \to y$ given by the counit.
    One defines a composition, using these evaluation maps, as the image of the identity of $z$ traced through
    \begin{align*}
        1_z \in \sfT(z,z) \to \sfT\big(\mcT(y,z) \ast y, z\big) \to& \sfT\big(\mcT(y,z) \ast \big(\mcT(x,y) \ast x\big), z \big) \\
            &\quad\iso \sfT\big(\big(\mcT(y,z) \tensor \mcT(x,y)\big) \ast x, z \big) \\
            &\quad\iso \sfS\big(\mcT(y,z) \tensor \mcT(x,y), \mcT(x,z)\big)
    \end{align*}
    The unit $i_x\colon \unit \to \mcT(x,x)$ is given via the isomorphism
    \begin{displaymath}
        1_x\in \sfT(x,x) \iso \sfT(\unit\ast x, x) \iso \sfS\big(\unit, \mcT(x,x)\big),
    \end{displaymath}
    which also shows that $\mcT_0 = \sfT$.  One then needs to check that the composition defined above is in fact associative and unital, which is a (mostly) straightforward exercise in diagram chasing. That $\mcT$ is copowered over $\sfS$ is immediate from the construction of the enrichment as copowers are just given by the $\sfS$-action.
\end{proof}

One can also give an interpretation of enriched functors in terms of actions. We next sketch a version of this sort of result which is relevant to our work. We don't give complete details (and are somewhat imprecise in the statement) as we will not really require this statement in the sequel. A detailed and more general treatment can be found in \cite{GordonPower}.  We recall that an $\sfS$-action provides us with unitors, that is, a natural isomorphism $l_x\colon 1\odot x = 1*x\to x$.
\begin{prop} \label{prop:action-enr-fun}
    Let $\sfF\colon\sfT \to \sfU$ be an exact coproduct preserving functor of triangulated categories admitting $\sfS$-actions as in Lemma \ref{lem:action-enr}.  The following are equivalent:
    \begin{enumerate}
        \item $\sfF$ is the underlying functor of a copower preserving $\sfS$-functor $\mcF\colon \mcT \to \mcU$.
        \item $\sfF$ commutes with the $\sfS$-action, via copowers, on $\sfT$ and $\sfU$, i.e., there are natural isomorphisms $s\odot \sfF x \stackrel{\gamma_{s,x}}\to \sfF(s\odot x)$ for all $s$ in $\sfS$ and $x$ in $\sfT$ verifying the following unitor and cocycle conditions (up to associators which we omit):
				\begin{align*}
                    l_{\sfF x} &= (\sfF l_x)\gamma_{\unit,x} \\
                    \gamma_{s\otimes s', x} &= \gamma_{s,s'\odot x} (s\odot\gamma_{s',x}).
				\end{align*}
    \end{enumerate}
\end{prop}
\begin{proof}
   If $\mcF$ is an $\sfS$-functor then one has natural comparison maps, as in Lemma~\ref{lem_copower_pres}, $s\odot \mcF x \to \mcF(s\odot x)$. Assuming that $\mcF$ preserves copowers just says that these natural maps are isomorphisms and, as they arise via the universal property of copowers, they satisfy the required unitor and cocycle conditions yielding compatibility of $\sfF$ with the action.
	
	Now let us suppose we are given an $\sfF$ together with coherent comparison maps $\gamma_{s,x}$ for all $s\in \sfS$ and $x\in \sfT$. We construct a candidate $\mcF$ by taking the same object assignment as for $\sfF$ and defining $\mcF_{x,y}$ to be the image of the composite
	\begin{displaymath}
	\commd{
        \node (1) at (0,0) {$\mcT(x,y)\odot \sfF x$};
        \node (2) at (5,0) {$\sfF(\mcT(x,y)\odot x)$};
        \node (3) at (8,0) {$\sfF y$};
    }{
        (1) edge node[auto] {$\gamma_{\mcT(x,y), x}$} (2)
        (2) edge node[auto] {$\sfF \varepsilon_{x,y}$} (3)
    }
	\end{displaymath}
	under the adjunction isomorphism
	\begin{displaymath}
        \sfU\big(\mcT(x,y)\odot \sfF x, \sfF y\big) \cong \sfS \big(\mcT(x,y), \mcU(\sfF x, \sfF y)\big).
	\end{displaymath}
	We then have to verify that $\mcF$ is in fact an $\sfS$-functor, i.e., the above morphisms are compatible with units and composition.
    
    From the unitor condition and naturality of $\gamma$, we obtain
    \begin{align*}
        l_{\sfF x} &= (\sfF l_x)\gamma_{\unit,x} \\
            &= (\sfF \varepsilon_x)\big(\sfF(i_x\odot1)\big)\gamma_{\unit,x} \\
            &= (\sfF \varepsilon_x) \gamma_{\mcT(x,x),x} (i_x\odot1)\colon \unit\odot \sfF x \to \sfF x,
    \end{align*}
    and passing through the adjunction yields $i_{\sfF x} = \mcF_{x,x} i_x\colon \unit \to \mcU(\sfF x,\sfF x)$, showing that $\mcF$ preserves units.

    To show that $\mcF$ preserves compositions, it is sufficient to show equality of the two adjunct morphisms $\mcT(y,z)\tensor\mcT(x,y) \odot \sfF x \to \sfF z$; namely that
    \begin{displaymath}
        \varepsilon_{\sfF y,\sfF z} (1\odot \varepsilon_{\sfF x,\sfF y}) (\mcF_{y,z} \odot \mcF_{x,y} \odot 1) = \varepsilon_{\sfF x, \sfF z} \mcF_{x,z} (\circ \odot 1).
    \end{displaymath}
    Using that $\varepsilon_{\sfF x,\sfF y}(\mcF_{x,y}\odot1) = (\sfF \varepsilon_{x,y})\gamma_{\mcT(x,y),x}$, naturality of $\gamma$ and the cocycle condition, we compute
    \begin{align*}
        &\varepsilon_{\sfF y,\sfF z} (1\odot \varepsilon_{\sfF x,\sfF y}) (\mcF_{y,z} \odot \mcF_{x,y} \odot 1) \\
        &\quad= (\sfF\varepsilon_{y,z}) \gamma_{\mcT(y,z),y} (1\odot\sfF\varepsilon_{x,y}) (1\odot\gamma_{\mcT(x,y),x}) \\
        &\quad= (\sfF\varepsilon_{y,z}) \sfF(1\odot \varepsilon_{x,y}) \gamma_{\mcT(y,z),\mcT(x,y)\odot y} (1\odot\gamma_{\mcT(x,y),x}) \\
        &\quad= (\sfF\varepsilon_{x,z}) \sfF(\circ\otimes1) \gamma_{\mcT(y,z)\tensor\mcT(x,y),x} \\
        &\quad= (\sfF\varepsilon_{x,z})\gamma_{\mcT(x,z),x}(\circ\odot1) \\
        &\quad= \varepsilon_{\sfF x, \sfF z} \mcF_{x,z} (\circ \odot 1).
    \end{align*}
    Thus $\mcF$ is an $\sfS$-functor which, by its construction, automatically preserves copowers.
\end{proof}

We now consider the canonical action of $\sfS$ on $\Modcat_\sfS A$ given by $s\tensor(x,\rho) = (s\tensor x, 1\tensor\rho)$.  Using the abstract result on actions giving enrichments we see that $\Modcat_\sfS A$ admits a corresponding enrichment which we shall denote by $\cModcat_\sfS A$, and whose hom objects we denote by $[-,-]_A$.

We note, as a particular consequence of the construction, that the functor $[A,-]_A\colon \Modcat_\sfS A \to \sfS$ arises as the right adjoint of $F_A = -\tensor A\colon \sfS \to \Modcat_\sfS A$, and therefore $[A,-]_A$ must be isomorphic to $U_A$, the forgetful functor.  This isomorphism can be made explicit by considering the adjunction isomorphism
\begin{displaymath}
    \Modcat_\sfS A(X\tensor A, X) \stackrel\sim\to \sfS(X,[A,X]_A),
\end{displaymath}
for any $A$-module $(X,\rho)$.  The isomorphism $X \stackrel\sim\to [A,X]_A$ is then given by $\rho^\flat$, the right adjunct to $\rho$ (where $\rho$ giving a map of right modules just expresses associativity of the right action of $A$ on $X$).  We will keep this notation in the sequel, denoting the right adjunct of $f$, say, by $f^\flat$ and the left adjunct by $f^\sharp$.  We omit the adjunction from the notation, as in all cases it will be clear from the context.

We also note that
\begin{displaymath}
    \circ\colon [Y,Z]_A \tensor [X,Y]_A \to [X,Z]_A
\end{displaymath}
by definition arises as the right adjunct to the composite
\begin{displaymath}
    [Y,Z]_A \tensor [X,Y]_A \tensor X \xto{1\tensor\varepsilon_{X,Y}} [Y,Z]_A \tensor Y \xto{\varepsilon_{Y,Z}} Z
\end{displaymath}
of counits.

The module structure is tightly connected with composition in the enrichment in the following way.
\begin{lem} \label{lem:rho-circ}
    Let $(X,\rho)$ be a right $A$-module.  Then the following diagram in $\sfS$ commutes:
    \begin{displaymath}
        \commd{
            \node (1) at (0,2) {$X\tensor A$};
            \node (2) at (6,2) {$X$};
            \node (3) at (0,0) {$[A,X]_A \tensor [A,A]_A$};
            \node (4) at (6,0) {$[A,X]_A$};
        }{
            (1) edge node[auto] {$\rho$} (2)
            (1) edge node[auto] {$\rho^\flat\tensor\mu^\flat$} node[auto,swap] {$\wr$} (3)
            (2) edge node[auto] {$\rho^\flat$} node[auto,swap] {$\wr$} (4)
            (3) edge node[auto] {$\circ$} (4)
        }
    \end{displaymath}
\end{lem}
\begin{proof}
    We consider the left adjuncts, and compute
    \begin{align*}
        (\rho^\flat\rho)^\sharp &= \varepsilon_{A,X} (\rho^\flat\rho\tensor1) \\
            &= \varepsilon_{A,X} (\rho^\flat\tensor1) (\rho\tensor1) \\
            &= \rho(\rho\tensor1),
    \end{align*}
    and, on the other hand
    \begin{align*}
        \big(\circ(\rho^\flat\tensor\mu^\flat)\big)^\sharp &= \varepsilon_{A,X}(\circ\tensor1)(\rho^\flat\tensor\mu^\flat\tensor1) \\
            &= \varepsilon_{A,X} (1\tensor\varepsilon_{A,A}) (\rho^\flat\tensor\mu^\flat\tensor1) \\
            &= \varepsilon_{A,X} (\rho^\flat\tensor\mu) \\
            &= \varepsilon_{A,X} (\rho^\flat\tensor1) (1\tensor\mu) \\
            &= \rho(1\tensor\mu).
    \end{align*}
    As $(X,\rho)$ is a right $A$-module these two expressions are equal.  Consequently the diagram commutes.
\end{proof}

Our aim is to determine when a functor between module categories of separable monoids is given by tensoring with a bimodule.  We will make precise what this means momentarily, but let us emphasize that it should at least be ``$\sfS$-linear''.  Thus by Lemma \ref{lem:action-enr} and Proposition \ref{prop:action-enr-fun} we are really making a statement about enriched functors.  

\subsection{Tensor products over separable monoids}
In order to make sense of the statement in Theorem \ref{thm:informaltensorrep} we need to define the tensor product over a separable monoid $A$ in $\sfS$.

First let us fix a right $A$-module $(X,\rho)$ and a left $A$-module $(Y,\lambda)$.  The endomorphism $e$
\begin{displaymath}
    X \tensor Y \stackrel\sim\to X \tensor\unit\tensor Y \xto{1\tensor\eta\tensor1} X \tensor A \tensor Y \xto{1\tensor\sigma\tensor1} X\tensor A\tensor A\tensor Y \xto{\rho \tensor \lambda} X \tensor Y
\end{displaymath}
is an idempotent, which one sees by considering the following commutative diagram (where we omit the intermediate objects, which can be deduced from the morphisms, for space reasons)
\begin{displaymath}
    \commd{
        \node (2) at (2,4) {$X\tensor Y$};
        \node (3) at (5,4) {};
        \node (4) at (8,4) {};
        \node (8) at (8,2) {};
        \node (10) at (8,0) {};
        \node (5) at (11,4) {};
        \node (9) at (11,2) {};
        \node (11) at (11,0) {};
        \node (6) at (14,4) {};
        \node (7) at (14,0) {$X\tensor Y$};
    }{
        (2) edge node[auto] {$1\tensor\sigma\eta\tensor1$} (3)
        (3) edge node[auto] {$1^{\tensor2}\tensor\eta\tensor1^{\tensor2}$} (4)
        (4) edge node[auto] {$1^{\tensor2}\tensor\sigma\tensor1^{\tensor2}$} (5)
        (5) edge node[auto] {$\rho\tensor1^{\tensor2}\tensor\lambda$} (6)
        (6) edge node[auto] {$\rho\tensor\lambda$} (7)
        (4) edge node[auto] {$1\tensor\mu\tensor1^{\tensor2}$} (8)
        (5) edge node[auto] {$1\tensor\mu\tensor1^{\tensor3}$} (9)
        (8) edge node[auto] {$1\tensor\mu\tensor1$} (10)
        (9) edge node[auto] {$1^{\tensor2}\tensor\mu\tensor1$} (11)
        (8) edge node[auto,swap] {$1\tensor\sigma\tensor1^{\tensor2}$} (9)
        (10) edge node[auto,swap] {$1\tensor\sigma\tensor1$} (11)
        (3) edge node[auto,swap] {$1^{\tensor4}$} (8)
        (2) edge node[auto,swap] {$1\tensor\eta\tensor1$} (10)
        (11) edge node[auto,swap] {$\rho\tensor\lambda$} (7)
    }
\end{displaymath}
where the composition along the top is $e^2$ and the one along the bottom is $e$.  As idempotents in $\sfS$ split, $\im(e)$ is a summand of $X\otimes Y$ and we define the tensor product over $A$ as
\begin{displaymath}
    X\tensor_A Y := \im(e),
\end{displaymath}
following Balmer \cite{Balmer_tt}.

We fix notation for the splitting as follows
\begin{displaymath}
    \commd{
        \node (1) at (0,0) {$\ker(e)$};
        \node (2) at (4,0) {$X\tensor Y$};
        \node (3) at (8,0) {$\im(e)$};
    }{
        ($(1.east)+(0,.1)$) edge node[auto] {$j$} ($(2.west)+(0,.1)$)
        ($(2.east)+(0,.1)$) edge node[auto] {$p$} ($(3.west)+(0,.1)$)
        ($(3.west)-(0,.1)$) edge node[auto] {$i$} ($(2.east)-(0,.1)$)
        ($(2.west)-(0,.1)$) edge node[auto] {$q$} ($(1.east)-(0,.1)$)
    }
\end{displaymath}
where both the upper and lower row are split exact triangles which satisfy
\begin{displaymath}
    pi = 1,\quad ip = e; \qquad qj = 1, \quad jq = 1-e.
\end{displaymath}

The next lemma shows that this definition of the tensor product over $A$ coincides with the usual one.
\begin{lem}[{\cite{Balmer_tt}}] \label{lem:coequalizer}
    The diagram
    \begin{displaymath}
        \commd{
            \node (1) at (0,0) {$X\tensor A \tensor Y$};
            \node (2) at (5,0) {$X\tensor Y$};
            \node (3) at (8,0) {$\im(e)$};
        }{
            ($(1.east)+(0,.1)$) edge node[auto] {$\rho\tensor1$} ($(2.west)+(0,.1)$)
            ($(1.east)-(0,.1)$) edge node[auto,swap] {$1\tensor\lambda$} ($(2.west)-(0,.1)$)
            (2) edge node[auto] {$p$} (3)
        }
    \end{displaymath}
    is a coequalizer in $\sfS$.
\end{lem}
\begin{proof}
    Composing with the monomorphism $i$, one sees that $p(\rho\tensor1) = p(1\tensor\lambda)$ is equivalent to $e(\rho\tensor1) = e(1\tensor\lambda)$.  The composite $e(\rho \otimes 1)$ is
    \begin{displaymath}
        X\tensor A\tensor Y \xto{1^{\tensor2}\tensor\sigma\eta\tensor1} X\tensor A\tensor A\tensor A\tensor Y \xto{\rho\tensor1^{\tensor3}} X\tensor A\tensor A\tensor Y \xto{\rho\tensor\lambda} X\tensor Y.
    \end{displaymath}
    Replacing $\rho\tensor1^{\tensor3}$ by $1\tensor\mu\tensor1^{\tensor2}$, and interchanging $\mu$ and $\sigma$ as per the definition of a separable monoid, we can rewrite this as
    \begin{displaymath}
        X\tensor A\tensor Y \xto{1\tensor\sigma\tensor1} X\tensor A\tensor A\tensor Y \xto{\rho\tensor\lambda} X\tensor Y.
    \end{displaymath}
    The composite $e(1\otimes \lambda)$ can also be rewritten this way and hence $e(\rho\otimes 1) = e(1\otimes \lambda)$.

    Next we show that the idempotent $e$ precisely detects when a morphism coequalizes $\rho\otimes 1$ and $1\otimes \lambda$.  More precisely, we claim that for a morphism $f\colon X\tensor Y \to Z$, $f(\rho\tensor1) = f(1\tensor\lambda)$ if and only if $f = fe$.
    
    First assume that $f(\rho\tensor1) = f(1\tensor\lambda)$.  Thus we have a commutative diagram
    \begin{displaymath}
        \commd{
            \node (1) at (-1,4) {$X\tensor Y$};
            \node (2) at (2,4) {};
            \node (3) at (5,4) {};
            \node (4) at (8,4) {};
            \node (5) at (11,4) {$X\tensor Y$};
            \node (6) at (5,2) {};
            \node (7) at (8,2) {$X\tensor Y$};
            \node (8) at (11,2) {$Z$};
        }{
            (1) edge node[auto] {$1\tensor\eta\tensor1$} (2)
            (2) edge node[auto] {$1\tensor\sigma\tensor1$} (3)
            (3) edge node[auto] {$\rho\tensor1^{\tensor2}$} (4)
            (4) edge node[auto] {$1\tensor\lambda$} (5)
            (5) edge node[auto] {$f$} (8)
            (4) edge node[auto] {$\rho\tensor1$} (7)
            (6) edge node[auto,swap] {$\rho\tensor1$} (7)
            (7) edge node[auto,swap] {$f$} (8)
            (3) edge node[auto] {$1\tensor\mu\tensor1$} (6)
            (2) edge node[auto,swap] {$1^{\tensor3}$} (6)
        }
    \end{displaymath}
    whose top row is $e$, showing that $fe = f$.  For the converse we use the first part of the proof and obtain
    \begin{displaymath}
        f(\rho\tensor1) = fe(\rho\tensor1) = fe(1\tensor\lambda) = f(1\tensor\lambda).
    \end{displaymath}
    
    Lastly, we need to show that the universal property holds under the assumption $fe = f$.  This equality can be rewritten as $fjq = f(1-e) = 0$, implying $fj = 0$ since $q$ is an epimorphism.  It follows that there is a unique morphism $\bar f\colon \im(e) \to Z$ such that $\bar fp = f$.
\end{proof}

The tensor product constructed above is a left adjoint in two variables to the internal homs we constructed to enrich the module categories over separable monoids. The next proposition makes this precise.

\begin{prop}
    Let $A$ and $B$ be separable monoids in $\sfS$.  Given a right $A$-module $X$, an $A$-$B$-bimodule $Y$ and a right $B$-module $Z$ there is an isomorphism
    \begin{displaymath}
        \Modcat_\sfS B( X \tensor_A Y, Z ) \iso \Modcat_\sfS A( X, [Y,Z]_B )
    \end{displaymath}
    natural in all three variables.
\end{prop}
\begin{proof}
    It is clear that $X\otimes_A Y$ is a right $B$-module. We begin by showing $[Y,Z]_B$ is indeed a right $A$-module. The left $A$-module structure of $Y$ produces a morphism in $\Modcat B$
    \begin{displaymath}
        [Y,Z]_B \tensor A \tensor Y \xto{1\tensor\lambda_Y} [Y,Z]_B \tensor Y \xto{\varepsilon_{Y,Z}} Z,
    \end{displaymath}
    which by adjunction yields the right $A$-module structure on $[Y,Z]_B$
    \begin{displaymath}
        [Y,Z]_B \tensor A \xto{\big(\varepsilon_{Y,Z}(1\tensor\lambda_Y)\big)^\flat} [Y,Z]_B.
    \end{displaymath}

    Fix a morphism $f\colon X\tensor_A Y \to Z$ in $\Modcat_\sfS B$.  Precomposing with the split epimorphism $p\colon X\tensor Y \to X\tensor_A Y$ we obtain a morphism
    \begin{displaymath}
        X \xto{(fp)^\flat} [Y,Z]_B
    \end{displaymath}
    in $\sfS$.  Showing that this is a morphism in $\Modcat_\sfS A$ amounts to showing the commutativity of
    \begin{displaymath}
        \commd{
            \node (1) at (0,2) {$X\tensor A$};
            \node (2) at (5,2) {$[Y,Z]_B\tensor A$};
            \node (3) at (0,0) {$X$};
            \node (4) at (5,0) {$[Y,Z]_B$.};
        }{
            (1) edge node[auto] {$(fp)^\flat\tensor1$} (2)
            (1) edge node[auto] {$\rho_X$} (3)
            (3) edge node[auto] {$(fp)^\flat$} (4)
            (2) edge node[auto] {$\big(\varepsilon_{Y,Z}(1\tensor\lambda_Y)\big)^\flat$} (4)
        }
    \end{displaymath}
    Taking left adjuncts reduces this to the following computation
    \begin{align*}
        \varepsilon(1\tensor\lambda_Y)\big((fp)^\flat\tensor1^{\tensor2}\big)
            &= \varepsilon\big((fp)^\flat\tensor1\big)(1\tensor\lambda_Y) \\
            &= fp(1\tensor\lambda_Y) \\
            &= fp(\rho_X\tensor1),
    \end{align*}
    where Lemma \ref{lem:coequalizer} yields the last equality.  It follows that the assignment $f \mapsto (fp)^\flat$ yields a morphism of $A$-modules.

    On the other hand, starting with a morphism $g\colon X \to [Y,Z]_B$ in $\Modcat_\sfS A$ it is clear that
    \begin{displaymath}
        X\tensor_A Y \stackrel i\to X\tensor Y \stackrel{g^\sharp}\to Z
    \end{displaymath}
    is a morphism of $B$-modules.
    
    We claim that these assignments are mutually inverse.  In one direction, we have
    \begin{displaymath}
        (fp)^{\flat\sharp}i = fpi = f,
    \end{displaymath}
    since $pi = 1$.  Lastly we show that
    \begin{displaymath}
        (g^\sharp ip)^\flat = (g^\sharp e)^\flat
    \end{displaymath}
    equals $g$, or equivalently that $g^\sharp e = g^\sharp$.  By Lemma \ref{lem:coequalizer} it suffices to show that $g^\sharp(1\tensor\lambda_Y) = g^\sharp(\rho_X\tensor1)$.  The commutativity of
    \begin{displaymath}
        \commd{
            \node (1) at (0,6) {$X\tensor A\tensor Y$};
            \node (2) at (8,6) {$X\tensor Y$};
            \node (3) at (4,4) {$[Y,Z]_B\tensor A\tensor Y$};
            \node (4) at (12,4) {$[Y,Z]_B\tensor Y$};
            \node (5) at (0,2) {$X\tensor Y$};
            \node (6) at (4,0) {$[Y,Z]_B\tensor Y$};
            \node (7) at (12,0) {$Z$};
        }{
            (1) edge node[auto] {$\rho_X\tensor1$} (2)
            (1) edge node[auto] {$g\tensor1^{\tensor2}$} (3)
            (2) edge node[auto] {$g\tensor1$} (4)
            (3) edge node[auto] {$\big(\varepsilon_{Y,Z}(1\tensor\lambda_Y)\big)^\flat\tensor1$} (4)
            (1) edge node[auto] {$1\tensor\lambda_Y$} (5)
            (3) edge node[auto] {$1\tensor\lambda_Y$} (6)
            (4) edge node[auto] {$\varepsilon_{Y,Z}$} (7)
            (6) edge node[auto] {$\varepsilon_{Y,Z}$} (7)
            (5) edge node[auto] {$g\tensor1$} (6)
        }
    \end{displaymath}
    yields this equality and so completes the argument.
\end{proof}

\subsection{A triangulated Eilenberg--Watts theorem}
We now prove one direction of the main result; Theorem~\ref{thm:tensorrep} below.
\begin{prop} \label{prop:tensorY}
    Let $A$ and $B$ be separable monoids in $\sfS$ and $Y$ an $A$-$B$-bimodule.  Then
    \begin{displaymath}
        -\tensor_A Y\colon \cModcat_\sfS A \to \cModcat_\sfS B
    \end{displaymath}
    is a copower preserving $\sfS$-functor.

    Moreover, the underlying functor is exact and preserves coproducts.
\end{prop}
\begin{proof}
    For ease of notation, let us denote this functor-to-be by $\mcG$.  For $A$-modules $M$ and $N$ we first construct a morphism
    \begin{displaymath}
        \mcG_{M,N}\colon [M,N]_A \to [M\tensor_A Y, N\tensor_A Y]_B.
    \end{displaymath}

    Consider the diagram
    \begin{displaymath}
    \commd{
        \node (1) at (0,6) {$[M,N]_A \tensor M \tensor A \tensor Y$};
        \node (2) at (6,6) {$N \tensor A \tensor Y$};
        \node (3) at (0,4) {$[M,N]_A \tensor M \tensor Y$};
        \node (4) at (6,4) {$N \tensor Y$};
        \node (5) at (0,2) {$[M,N]_A \tensor M\tensor_A Y$};
        \node (6) at (6,2) {$N \tensor_A Y$};
    }{
        (1) edge node[auto] {$\varepsilon_{M,N}\tensor1\tensor1$} (2)
        (3) edge node[auto] {$\varepsilon_{M,N}\tensor1$} (4)
        (5) edge[densely dashed] node[auto] {$u_{M,N}$} (6)
        (1) edge node[auto,swap] {$1\tensor\rho\tensor1-1\tensor1\tensor\lambda$} (3)
        (3) edge node[auto] {} (5)
        (2) edge node[auto] {$\rho\tensor1 - 1\tensor\lambda$} (4)
        (4) edge node[auto] {} (6)
    }
    \end{displaymath}
    where the upper square commutes by naturality of the counit and the columns are coequalizers.  Since the composite along the top then right edge is $0$, there is thus a unique morphism of right $B$-modules $u_{M,N}$ making the lower square commute.  From this we obtain the adjunct
    \begin{displaymath}
        \mcG_{M,N} = u_{M,N}^\flat\colon [M,N]_A \to [M\tensor_A Y, N\tensor_A Y]_B.
    \end{displaymath}

    Now assume that $M=N$.  The composition
    \begin{displaymath}
        M \tensor Y \iso \unit\tensor M \tensor Y \xto{1_M^\flat\tensor1\tensor1} [M,M]_A \tensor M \tensor Y \xto{\varepsilon_{M,M} \tensor 1} M \tensor Y
    \end{displaymath}
    is the identity on $M\tensor Y$.  It follows that the composition
    \begin{displaymath}
        M \tensor_A Y \iso \unit\tensor M \tensor_A Y \xto{1_M^\flat\tensor1} [M,M]_A \tensor M \tensor_A Y \xto{u_{M,M}} M \tensor_A Y
    \end{displaymath}
    is the identity on $M\tensor_A Y$.  Consequently, $\mcG$ preserves the unit, i.e.,
    \begin{displaymath}
        \unit \to [M,M]_A \xto{\mcG_{M,M}} [M\tensor_A Y, M\tensor_A Y]_B
    \end{displaymath}
    is the unit $\unit \to [M\tensor_A Y, M\tensor_A Y]_B$.

    One shows that $\mcG$ is compatible with composition in $\cModcat_\sfS A$ and $\cModcat_\sfS B$ by a similar argument.  Thus $\mcG$ is an $\sfS$-functor.

    We now show that this functor preserves copowers.  Recall from Lemma~\ref{lem:action-enr} that both $\cModcat_\sfS A$ and $\cModcat_\sfS B$ are copowered over $\sfS$ so this statement is reasonable.
	Preservation of copowers follows from the fact that associativity of the tensor product in $\sfS$ descends to summands, i.e.,
    \begin{displaymath}
        s\tensor(M\tensor_A Y) \iso (s\tensor M)\tensor_A Y,
    \end{displaymath}
    where this isomorphism is the canonical morphism of Lemma \ref{lem_copower_pres}.

    Lastly, assume that $A$ and $B$ are separable.  The underlying functor $\mcG_0$ is a summand of the exact coproduct preserving functor $-\otimes_AY$.  Triangles (respectively coproducts) in both $\Modcat_\sfS A$ and $\Modcat_\sfS B$ are characterized by being triangles (respectively coproducts) in $\sfS$, and so the result follows from exactness and coproduct preservation of $\tensor$ in $\sfS$.
\end{proof}

We now embark on the proof that the properties of the previous proposition are sufficient to guarantee that the functor is isomorphic to a tensor product over $A$.

\begin{prop} \label{prop:FA-bimod}
    Let $\mcF\colon \cModcat_\sfS A \to \cModcat_\sfS B$ be an $\sfS$-functor.  The object $\mcF A$ is an $A$-$B$-bimodule.
\end{prop}
\begin{proof}
    As $\mcF A$ is an object of $\Modcat_\sfS B$, it is a right $B$-module via some $\rho_{\mcF A}\colon \mcF A\tensor B \to \mcF A$.  Furthermore, as $\mcF$ is enriched there is a morphism in $\sfS$
    \begin{displaymath}
        \mcF_{A,A}\colon [A,A]_A \to [\mcF A,\mcF A]_B,
    \end{displaymath}
    which in turn gives rise to the morphism
    \begin{displaymath}
        \lambda_{\mcF A}\colon A \tensor \mcF A \xto{\mu^\flat\tensor1} [A,A]_A \tensor \mcF A \xto{\mcF_{A,A}^\sharp} \mcF A,
    \end{displaymath}
    where $\mu$ denotes the multiplication on $A$, in $\Modcat_\sfS B$.  We claim this endows $\mcF A$ with a left $A$-module structure; we need only check the commutativity of the following diagram:
    \begin{displaymath}
    \commd{
        \node (1) at (0,2) {$A\tensor A\tensor \mcF A$};
        \node (2) at (4,2) {$A\tensor \mcF A$};
        \node (3) at (0,0) {$A\tensor \mcF A$};
        \node (4) at (4,0) {$\mcF A$};
    }{
        (1) edge node[auto] {$1\tensor\lambda_{\mcF A}$} (2)
        (1) edge node[auto] {$\mu\tensor1$} (3)
        (2) edge node[auto] {$\lambda_{\mcF A}$} (4)
        (3) edge node[auto] {$\lambda_{\mcF A}$} (4)
    }
    \end{displaymath}
    Via the adjunction, commutativity of this diagram is equivalent to that of
    \begin{displaymath}
    \commd{
        \node (1) at (0,2) {$A\tensor A$};
        \node (2) at (5,2) {$[A,A]_A\tensor[A,A]_A$};
        \node (3) at (12,2) {$[\mcF A,\mcF A]_B\tensor[\mcF A,\mcF A]_B$};
        \node (4) at (0,0) {$A$};
        \node (5) at (5,0) {$[A,A]_A$};
        \node (6) at (12,0) {$[\mcF A,\mcF A]_B.$};
    }{
        (1) edge node[auto] {$\mu^\flat\tensor\mu^\flat$} (2)
        (2) edge node[auto] {$\mcF_{A,A}\tensor\mcF_{A,A}$} (3)
        (4) edge node[auto] {$\mu^\flat$} (5)
        (5) edge node[auto] {$\mcF_{A,A}$} (6)
        (1) edge node[auto] {$\mu$} (4)
        (2) edge node[auto] {$\circ$} (5)
        (3) edge node[auto] {$\circ$} (6)
    }
    \end{displaymath}
    This diagram is readily seen to commute:  the commutativity of the first square is Lemma \ref{lem:rho-circ}, while the second commutes since $\mcF$ is an enriched functor.

    It remains to check that the left and right module structures are compatible, i.e., that
    \begin{displaymath}
    \commd{
        \node (1) at (0,2) {$A\tensor \mcF A \tensor B$};
        \node (2) at (4,2) {$\mcF A\tensor B$};
        \node (3) at (0,0) {$A\tensor \mcF A$};
        \node (4) at (4,0) {$\mcF A$};
    }{
        (1) edge node[auto] {$\lambda_{\mcF A}\tensor1$} (2)
        (1) edge node[auto] {$1\tensor\rho_{\mcF A}$} (3)
        (2) edge node[auto] {$\rho_{\mcF A}$} (4)
        (3) edge node[auto] {$\lambda_{\mcF A}$} (4)
    }
    \end{displaymath}
    commutes.  This is precisely the statement that $\lambda$ is a morphism in $\Modcat_\sfS B$, which is true by construction.  Hence $\mcF A$ is an $A$-$B$-bimodule as claimed.
\end{proof}

Let $\mcF\colon \cModcat_\sfS A \to \cModcat_\sfS B$ be an $\sfS$-functor.  In order to prove the theorem, we must first exhibit an enriched natural transformation
\begin{displaymath}
    \a\colon -\tensor_A \mcF A \to \mcF.
\end{displaymath}
The next two lemmas dispose of this task.

\begin{lem} \label{lem:compare}
    Let $(M,\rho_M)$ be a right $A$-module.  There is a canonical morphism $M\tensor\mcF A \to \mcF M$ in $\cModcat_\sfS B$ such that the composite
    \begin{displaymath}
        M \tensor A \tensor \mcF A \xto{\rho_M\tensor1 - 1\tensor\lambda_{\mcF A}} M\tensor\mcF A \to \mcF M
    \end{displaymath}
    is zero.
\end{lem}
\begin{proof}
    The morphism is given as the composite
    \begin{displaymath}
        M\tensor\mcF A \xto{\rho_M^\flat\tensor1} [A,M]_A \tensor \mcF A \xto{\mcF_{A,M}^\sharp} \mcF M,
    \end{displaymath}
    and we have previously (see Lemma~\ref{lem:rho-circ}) established the commutativity of the left square in the following diagram:
    \begin{displaymath}
    \commd{
        \node (1) at (0,2) {$M \tensor A \tensor \mcF A$};
        \node (2) at (8,2) {$M\tensor\mcF A $};
        \node (3) at (12,2) {$\mcF M$};
        \node (4) at (0,0) {$[A,M]_A \tensor [A,A]_A \tensor \mcF A$};
        \node (5) at (8,0) {$[A,M]_A \tensor\mcF A$};
        \node (6) at (12,0) {$\mcF M$};
    }{
        (1) edge node[auto] {$\rho_M\tensor1 - 1\tensor\lambda_{\mcF A}$} (2)
        (2) edge node[auto] {} (3)
        (4) edge node[auto] {$\circ\tensor1 - 1\tensor\mcF_{A,A}^\sharp$} (5)
        (5) edge node[auto] {$\mcF_{A,M}^\sharp$} (6)
        (1) edge node[auto] {$\rho_M^\flat\tensor\mu^\flat\tensor1$} (4)
        (2) edge node[auto] {$\rho_M^\flat\tensor1$} (5)
        (3) edge[thick,double distance=2pt,-] node[auto] {} (6)
    }
    \end{displaymath}
    in which the vertical arrows are all isomorphisms. It is enough to show that the composite of the two morphisms in the bottom row is zero (in fact going down and then along the bottom row is the map we want on the nose).  This follows from following commutative diagram
    \begin{displaymath}
    \commd{
        \node (1) at (0,4) {$[A,M]_A \tensor [A,A]_A \tensor \mcF A$};
        \node (2) at (0,2) {$[\mcF A,\mcF M]_B \tensor [\mcF A,\mcF A]_B \tensor \mcF A$};
        \node (3) at (0,0) {$[\mcF A,\mcF M]_B \tensor \mcF A$};
        \node (4) at (8,4) {$[A,M]_A \tensor \mcF A$};
        \node (5) at (8,2) {$[\mcF A,\mcF M]_B \tensor \mcF A$};
        \node (6) at (8,0) {$\mcF M$};
    }{
        (1) edge node[auto] {$\mcF_{A,M}\tensor\mcF_{A,A}\tensor1$} (2)
        (2) edge node[auto] {$1\tensor\varepsilon_{\mcF A}$} (3)
        (4) edge node[auto] {$\mcF_{A,M}\tensor1$} (5)
        (5) edge node[auto] {$\varepsilon_{\mcF A}$} (6)
        (1) edge node[auto] {$\circ\tensor1$} (4)
        (2) edge node[auto] {$\circ\tensor1$} (5)
        (3) edge node[auto] {$\varepsilon_{\mcF A}$} (6)
    }
    \end{displaymath}
    which shows that $\mcF_{A,M}^\sharp(1\tensor\mcF_{A,A}^\sharp)$ equals $\mcF_{A,M}^\sharp(\circ\tensor1)$.
\end{proof}

By the construction of $M\tensor_A \mcF A$ (as a cokernel) there is therefore a unique factorization in $\Modcat_\sfS B$ of $M\tensor \mcF A \to \mcF M$ via a morphism
\begin{displaymath}
    \a_M\colon M\tensor_A \mcF A \to \mcF M.
\end{displaymath}

\begin{lem}\label{lem:wattsnat}
    The $\a_M$ are the components of an enriched natural transformation
    \begin{displaymath}
        \a \colon -\otimes_A \mcF A \to \mcF.
    \end{displaymath}
\end{lem}
\begin{proof}
    Recall from Section~\ref{sec:prelim} that naturality of $\a$ is expressed by the commutativity of the diagram
    \begin{displaymath}
    \commd{
        \node (1) at (2,0) {$[M,N]_A\otimes \unit$};
        \node (2) at (12,0) {$[\mcF M, \mcF N]_B\otimes [M\otimes_A \mcF A, \mcF M]_B$};
        \node (3) at (1,2) {$[M,N]_A$};
        \node (4) at (13,2) {$[M\otimes_A \mcF A, \mcF N]_B$};
        \node (5) at (2,4) {$\unit\otimes[M,N]_A$};
        \node (6) at (12,4) {$[N\otimes \mcF A, \mcF N]_B \otimes [M\otimes_A \mcF A, N\otimes_A \mcF A]_B$};
    }{
        (3) edge node[auto] {} (1)
        (3) edge node[auto] {} (5)
        (1) edge node[auto,swap] {$\mcF_{M,N}\otimes\a_M^\flat$} (2)
        (5) edge node[auto] {$\a_N^\flat\otimes (-\otimes_A\mcF A)_{M,N}$} (6)
        (2) edge node[auto,swap] {$\circ$} (4)
        (6) edge node[auto] {$\circ$} (4)
    }
    \end{displaymath}
    for all $M,N \in \cModcat_\sfS A$.  Since we have a better grasp on the $\a_M$s than we have on their adjuncts, it is convenient to rewrite this diagram.  Via the adjunction, naturality can also be expressed by the commutativity of
    \begin{displaymath}
        \commd{
            \node (1) at (0,2) {$[M,N]_A \otimes M\otimes_A \mcF A$};
            \node (2) at (0,0) {$[M,N]_A\otimes \mcF M$};
            \node (3) at (6,2) {$N\otimes_A \mcF A$};
            \node (4) at (6,0) {$\mcF N$};
        }{
            (1) edge node[auto] {$1 \otimes \a_M$} (2)
            (1) edge node[auto] {$\varepsilon_{M,N}\otimes_A 1$} (3)
            (2) edge node[auto] {$\mcF^\sharp_{M,N}$} (4)
            (3) edge node[auto] {$\a_N$} (4)
        }
    \end{displaymath}
    To see this commutes consider the following expanded diagram:
    \begin{displaymath}
        \commd{
            \node (7) at (0,6) {$[M,N]_A \otimes M\otimes_A \mcF A$};
            \node (8) at (8,6) {$N\otimes_A \mcF A$};
            \node (1) at (0,4) {$[M,N]_A \otimes M\otimes \mcF A$};
            \node (2) at (0,2) {$[M,N]_A \otimes [A,M]_A\otimes \mcF A$};
            \node (3) at (0,0) {$[M,N]_A\otimes \mcF M$};
            \node (4) at (8,4) {$N\otimes \mcF A$};
            \node (5) at (8,2) {$[A,N]_A \otimes \mcF A$};
            \node (6) at (8,0) {$\mcF N$};
        }{
            (7) edge node[auto] {} (1)
            (8) edge node[auto] {} (4)
            (7) edge node[auto] {$\varepsilon_{M,N}\otimes_A 1$} (8)
            (1) edge node[auto] {$1\tensor\rho_M^\flat\tensor1$} (2)
            (2) edge node[auto] {$1\otimes\mcF^\sharp_{A,M}$} (3)
            (1) edge node[auto] {$\varepsilon_{M,N}\otimes1$} (4)
            (2) edge node[auto] {$\circ\otimes1$} (5)
            (3) edge node[auto] {$\mcF^\sharp_{M,N}$} (6)
            (4) edge node[auto] {$\rho_N^\flat\tensor1$} (5)
            (5) edge node[auto] {$\mcF^\sharp_{A,N}$} (6)
        }
    \end{displaymath}
    The upper two squares commute by naturality of the counit. The adjunct of the bottom square just expresses the fact that $\mcF$ is an enriched functor and so it also commutes.

    Thus the outer rectangle commutes, proving that $\a$ is an enriched natural transformation.
\end{proof}

It remains to show that each $\a_M$ is an isomorphism in $\Modcat_\sfS B$.  The following lemma does most of the work.

\begin{lem}\label{lem:wattsunit}
    For any left $A$-module $(Y,\lambda)$ there is a canonical isomorphism 
		\begin{displaymath}
		A \tensor_A Y \iso Y.
		\end{displaymath}
    Moreoever, the component of $\a$ at $A$,
    \begin{displaymath}
        \a_A \colon A\otimes_A \mcF A \to \mcF A,
    \end{displaymath}
    is precisely this canonical map for $Y= \mcF A$ and hence is an isomorphism.
\end{lem}
\begin{proof}
    Consider the morphisms
    \begin{displaymath}
        \bar\lambda\colon A\tensor_A Y \stackrel i\to A\tensor Y \stackrel\lambda\to Y
    \end{displaymath}
    and
    \begin{displaymath}
        \bar\eta\colon Y \stackrel \sim\to \unit\otimes Y \xto{\eta\tensor1} A\tensor Y \stackrel p\to A\tensor_A Y,
    \end{displaymath}
    where the notation is as in the definition of $\otimes_A$ (see the diagram before Lemma~\ref{lem:coequalizer}). Our claim is that $\bar\lambda$ and $\bar\eta$ are inverse isomorphisms. We recall that since $e^2 = e$, we have $e(1\tensor\lambda) = e(\mu\tensor1)$, so that in one direction we have
    \begin{align*}
        i\bar\eta\bar\lambda p &= e(1\tensor\lambda)(\eta\tensor1\tensor1)e \\
            &= e(\mu\tensor1)(\eta\tensor1\tensor1)e \\
            &= e^2 = e = ip.
    \end{align*}
    Since $i$ is a monomorphism and $p$ is an epimorphism, we conclude that $\bar\eta\bar\lambda = 1$.

    Going the other way, we have
    \begin{align*}
        \bar\lambda\bar\eta &= \lambda e(\eta\tensor1) \\
            &= \lambda(\mu\tensor\lambda)(1\tensor\sigma\tensor1)(1\tensor\eta\tensor1)(\eta\tensor1) \\
            &= \lambda(1\tensor\lambda)(\sigma\tensor1)(\mu\tensor1)(1\tensor\eta\tensor1)(\eta\tensor1) \\
            &= \lambda(1\tensor\lambda)(1\tensor\eta\tensor1)(\eta\tensor1) \\
            &= 1,
    \end{align*}
		proving the first part of the statement.

    For the second claim, we simply note that the component of $\a$ at $A$ can be written as
    \begin{displaymath}
        A \tensor_A \mcF A \stackrel i\to A\tensor \mcF A \xto{\lambda_{\mcF A}} \mcF A. \qedhere
    \end{displaymath}
\end{proof}

Up to this point, we have not fully utilized our assumptions on the $\sfS$-functor $\mcF$.  Now, however, we will use that the underlying functor of $\mcF$, which we denote by $\sfF$, is exact and preserves coproducts.  With these hypotheses we can use the standard trick to prove our analog of the Eilenberg--Watts theorem.

\begin{lem}\label{lem:wattsloc}
    The full subcategory
    \begin{displaymath}
        \sfM = \{M\in \Modcat_\sfS A \mid \text{$\a_M$ is an isomorphism}\}
    \end{displaymath}
    is localizing.
\end{lem}
\begin{proof}
    The underlying natural transformation of $\a$, whose components are just the $\a_M$, is a natural transformation between the exact coproduct preserving functors $(-\otimes_A \mcF A)_0$ and $\sfF$. Thus, as in Lemma~\ref{lem:brownloc}, $\sfM$ is localizing.
\end{proof}

\begin{lem} \label{lem:Agen}
    Suppose $\sfS$ is a compactly generated tensor triangulated category such that the unit $\unit$ is a compact generator.  Then if $A$ is a separable monoid in $\sfS$, the regular representation $(A,\mu)$ is a compact generator for $\Modcat_\sfS A$.
\end{lem}
\begin{proof}
Consider the free-forgetful adjunction
\begin{displaymath}
\commd{
        \node (1) at (0,0) {$\sfS$};
        \node (2) at (3,0) {$\Modcat_\sfS A$};
    }{
        ($(1.east)+(0,.1)$) edge node[auto] {$F_A$} ($(2.west)+(0,.1)$)
        ($(2.west)-(0,.1)$) edge node[auto] {$U_A$} ($(1.east)-(0,.1)$)
    }
\end{displaymath}
where $F_A = -\otimes A$ and $U_A$ just forgets the action. Since $U_A$ preserves coproducts the functor $F_A$ sends compacts to compacts by \cite[Theorem~5.1]{NeeGrot}. Thus $F_A(\unit) = (A,\mu)$ is compact in $\Modcat_\sfS A$.

To see that it generates observe that for $M\in \Modcat_\sfS A$ we have isomorphisms
\begin{displaymath}
    \sfS\big(\S^i \unit, U_A(M)\big) \cong \Modcat_\sfS A\big(\S^iF_A(\unit), M\big) \cong \Modcat_\sfS A(\S^iA, M).
\end{displaymath}
If $M$ is non-zero then clearly $U_A(M)$ is also non-zero and so we can find a non-zero morphism $\S^i \unit \to U_A(M)$ for some $i\in \mathbb{Z}$ by the assumption that $\unit$ generates $\sfS$. Using the above isomorphisms we find a non-zero map from $\S^iA \to M$ in $\Modcat_\sfS A$ and so $A$ generates $\Modcat_\sfS A$ as claimed.
\end{proof}

We now come to the main theorem of this section, our (enriched) triangulated version of the Eilenberg--Watts theorem.
\begin{thm} \label{thm:tensorrep}
    Let $\sfS$ be a compactly generated tensor $\infty$-triangulated category such that the tensor unit $\unit$ is a compact generator.	Let $A$ and $B$ be separable monoids in $\sfS$ and let $\mcF$ be an $\sfS$-functor
    \begin{displaymath}
        \mcF\colon \cModcat_\sfS A \to \cModcat_\sfS B.
    \end{displaymath}
    Then $\mcF$ preserves copowers and the underlying functor $\sfF$ is exact and preserves coproducts if and only if there exists an $A$-$B$-bimodule $Y$ such that
    \begin{displaymath}
        \mcF \iso -\tensor_A Y.
    \end{displaymath}

    Furthermore, if this is the case then $Y \cong \mcF A$.
\end{thm}
\begin{proof}
    We have already proved one direction in Proposition~\ref{prop:tensorY}, namely that given an $A$-$B$-bimodule $Y$ the functor $-\otimes_A Y$ verifies the required properties.

    Suppose on the other hand that $\mcF$ is an $\sfS$ functor as in the statement. By Lemma~\ref{lem:wattsnat} there is an enriched natural transformation
    \begin{displaymath}
        \a\colon {-\otimes_A \mcF A} \to \mcF \\
    \end{displaymath}
    and we claim this is a natural isomorphism. From Lemma \ref{lem:wattsloc} we know that the full subcategory $\sfM$ consisting of those objects $M\in \Modcat_\sfS A$ for which $\a_M$ is an isomorphism is localizing in $\Modcat_\sfS A$.  We learned in Lemma~\ref{lem:wattsunit} that $\a_A$ is an isomorphism and hence $A$ is an object of this localizing subcategory.  Thus by the previous lemma we have
    \begin{displaymath}
        \Modcat_\sfS A = \langle A \rangle \subseteq \sfM,
    \end{displaymath}
    proving that $\a$ is indeed a natural isomorphism.
\end{proof}

\begin{rem}
    The theorem provides further moral support for the notion that tensor products on, and enrichments of, triangulated categories should be somehow connected to more conventional enhancements, for instance by dg-categories or model categories. The vague idea is that, in the case where one has a suitable enhancement, tensoring with a bimodule object should automatically admit an enhancement and so showing a functor is enriched (as in the theorem) also shows it admits an enhancement. On the other hand there are various results showing that enhanced functors are given by the appropriate notion of bimodules and so their induced functors on homotopy categories should lift to enriched functors.
\end{rem}

We finish by sketching a relatively simple application of the theorem.

\begin{cor}
    Let $k$ be a field and let $A$ and $B$ be separable $k$-algebras.  Then any coproduct preserving exact $k$-linear functor $\sfF\colon \D(A) \to \D(B)$ between the unbounded derived categories is given by tensoring with the $A$-$B$-bimodule $\sfF A$.
\end{cor}
\begin{proof}
    Being flat over $k$, $A$ and $B$, viewed as stalk complexes, are separable monoids in $\D(k)$.  By \cite[Theorem~6.5]{Balmer_sep}, there is a canonical equivalence $\D(A) \simeq \Modcat_{\D(k)} A$ as $\infty$-triangulated categories (and similarly for $B$).  
    
    In order to apply Theorem \ref{thm:tensorrep} we thus need to produce a copower preserving $\D(k)$-functor
    \begin{displaymath}
        \mcF\colon \cModcat_{\D(k)} A \to \cModcat_{\D(k)} B
    \end{displaymath}
    lifting $\sfF$. By Proposition~\ref{prop:action-enr-fun} it is sufficient to produce coherent comparison maps
    \begin{displaymath}
        x\odot \sfF M \stackrel\sim\to \sfF(x\odot M)
    \end{displaymath}
    for all $x\in \D(k)$ and $M\in \Modcat_{\D(k)}A$. As every object in $\D(k)$ is a sum of suspensions of copies of $k$, copowers by objects of $\D(k)$ are just given by taking direct sums and suspensions. Since $\sfF$ preserves the suspension and coproducts one can construct such a family of coherent comparison maps using $\S \sfF \cong \sfF \S$ and the universal property of coproducts in the evident way. The cocycle condition is essentially for free due to the universal property of coproducts. Thus there is an enriched lift $\mcF$ of $\sfF$ to which we can apply the theorem and we conclude that
    \begin{displaymath}
        \sfF \cong -\otimes_A \sfF A
    \end{displaymath}
    by taking underlying functors.
\end{proof}

\begin{rem}
    We have been unable to extend the above corollary to more general settings while maintaining reasonable hypotheses. Although the condition of Proposition~\ref{prop:action-enr-fun} appears very mild it seems very difficult to find checkable assumptions that allow one to verify it in abstract settings. However, we believe that in concrete situations the theorem could be of use. Moreover, provided one restricts from the beginning to the enriched setting it should also allow one to develop some Morita theory for separable monoids and perfom Tannaka type reconstruction at the level of enriched triangulated categories.
\end{rem}

\bibliography{enrichedrep}
\bibliographystyle{amsplain}

\end{document}